\documentclass[11pt, a4paper]{amsart}
\usepackage{amsmath}
\usepackage{amsfonts}
\usepackage{graphicx}
\usepackage{framed}
\usepackage{amssymb}
\usepackage{esvect}
\usepackage{paralist}

\usepackage{etex}
\usepackage{latexsym}
\usepackage[english]{babel}
\usepackage[utf8x]{inputenc}

\theoremstyle{plain}
\newtheorem{theorem}{Theorem}[section]
\newtheorem{lemma}[theorem]{Lemma}
\newtheorem{proposition}[theorem]{Proposition}
\theoremstyle{definition}
\newtheorem{definition}[theorem]{Definition}
\newtheorem{example}[theorem]{Example}
\newtheorem{remark}[theorem]{Remark}

\numberwithin{equation}{section}

\newcommand{\R}{\mathbb{R}}
\newcommand{\N}{\mathbb{N}}

\newcommand{\de}{\partial}

\renewcommand{\AA}{\mathcal{A}}

\newcommand{\bldu}{\mathbf{u}}

\newcommand{\bldz}{\mathbf{z}}
\newcommand{\GG}{\mathcal{G}}
\renewcommand{\d}{\mathrm{d}}

\newcommand{\PP}{\mathcal{P}}
\newcommand{\BB}{\mathcal{B}}

\newcommand{\f}{\mathfrak{f}}

\usepackage{color}

\begin{document}

\title[Positive solutions of fractional Laplacian systems]{Nonzero positive solutions of fractional Laplacian systems with functional terms}

\date{\today}

\author[S. Biagi]{Stefano Biagi}
\address{Stefano Biagi, Dipartimento di Matematica,
Politecnico di Milano, Via Bonardi~9, 20133 Milano, Italy}%
\email{stefano.biagi@polimi.it}%

\author[A. Calamai]{Alessandro Calamai}
\address{Alessandro Calamai, Dipartimento di Ingegneria Civile, Edile e Ar\-chi\-tet\-tu\-ra,
Universit\`a Politecnica delle Marche, Via Brecce Bianche, 60131 Ancona, Italy}%
\email{calamai@dipmat.univpm.it}%

\author[G. Infante]{Gennaro Infante}
\address{Gennaro Infante, Dipartimento di Matematica e Informatica, Universit\`{a} del\-la
Calabria, 87036 Arcavacata di Rende, Cosenza, Italy}%
\email{gennaro.infante@unical.it}%

\begin{abstract} 
We study the existence of non-zero positive solutions 
of  a class of systems of differential equations driven by fractional powers of the Laplacian.
Our approach is based on the notion of fixed point index, 
and allows us to deal with non-local functional weights and functional boundary conditions. We present two examples to shed light on the type of functionals and growth conditions that can be considered with our approach.

\end{abstract}



\maketitle

\section{Introduction and preliminaries}
The main aim of the present paper is to establish some
existence and non-existence results for
`functional' Dirichlet problems driven by fractional
powers of the classical Laplace operator.
More precisely, if $m\geq 1$ is a fixed
natural num\-ber, we shall be concerned with Dirichlet problems
of the following form
\begin{equation} \label{eq.PBIntro}
 \begin{cases}
  (-\Delta)^{s_i}u_i = \lambda_i\,f_i(x,\mathbf{u},\PP_i[\mathbf{u}]) & \text{in $\Omega$
  \qquad\qquad ($i = 1,\ldots,m$)}, \\
  u_i \equiv \eta_i\,\zeta_i(x)\,\BB_i[\mathbf{u}],
  & \text{in $\R^n\setminus\Omega$ \qquad ($i = 1,\ldots,m$)}, \\
  \mathbf{u}\gneqq 0 & \text{in $\R^n$},
 \end{cases}
\end{equation}
where $\Omega\subseteq\R^n$ is a fixed open set,
$\mathbf{u} = (u_1,\ldots,u_m)$
and, for $i = 1,\ldots,m$, 
\begin{itemize}
 \item $f_i$ is a real-valued function defined on $\Omega\times\R^m\times\R$;
 \item $\lambda_i,\,\eta_i$ are non-negative parameters;
 \item $\zeta_i$ is a sufficiently regular, real-valued function defined on $\R^n$;
 \item $\PP_i,\,\BB_i$ are suitable functionals to be defined later.
\end{itemize} 
Moreover, $s_1,\ldots,s_m\in (0,1)$ and
$(-\Delta)^{s_i}$ denotes the standard fractional Laplace operator
of order $s_i$, which is the non-local operator defined as
$$(-\Delta)^{s_i}v(x) = c_{n,s_i}\cdot
\mathrm{P.V.}\int_{\R^n}\frac{v(x)-v(y)}{|x-y|^{n+2{s_i}}}\,\d y.$$
Here, $c_{n,s_i} > 0$ is the `normalization' constant defined as
$$c_{n,s_i} := \bigg(\int_{\R^n}\frac{1-\cos(y_1)}{|y|^{n+2s_i}}\,\d y\bigg)^{-1}.$$ 
Notice that, in addition to the fractional differential operators, in system \eqref{eq.PBIntro} 
other non-local terms occur, both in the differential  equations (having the role of \emph{non-local 
functional weights}) and in the boundary conditions (BCs for short).
In particular, we are interested in the existence/non-existence of 
\emph{positive} solutions of 
\eqref{eq.PBIntro}, and our approach is based on the classical notion of fixed point index in cones.
We work in the Banach space of the bounded continuous
$\R^m$-valued functions defined in $\R^n$, namely 
$$
    \mathbb{X}  := \left\{u\in C(\R^n;\R^m):\,\sup_{\R^n}|u_i|<\infty\,\,\text{
   for all $i = 1,\ldots,m$}\right\},
$$
endowed with the supremum norm; accordingly, since we are interested
in non-zero positive solutions,
we look for solutions of \eqref{eq.PBIntro} lying in the cone
$$
   P := \left\{\mathbf{u}\in \mathbb{X}:\,\text{$u_i\geq 0$ on $\R^n$ for every $i = 1,\ldots,m$}\right\}.
$$
In view of these facts, 
it is natural to assume that that the real-valued operators $\PP_i,\,\BB_i$ are defined on $\mathbb{X}$
(for all $i = 1,\ldots,m$). \medskip

As it is by now well-known, equations involving the fractional Laplacian arise in several applications; because of this, they have been extensively studied in the last decades by many authors, we provide as a reference the comprehensive survey~\cite{DNPV}.
Among others, let us mention here the equations, driven by the fractional Laplacian,
which have additional non-local terms and are often
referred to as \emph{Kirchhoff-type equations}.
For instance, Kirchhoff-type equations on bounded domains have been recently studied  in
\cite{AmbSer19,Chen15, FV14, SunTeng14},
while \emph{systems}  on bounded domains are investigated, e.g., in \cite{CM2020}.
Moreover, Kirchhoff-type equations on the whole of $\R^n$ have been studied in
\cite{AlMi16, Amb19, AmbIs18,Secchi}.
As regards \emph{concrete} `real-word' applications,
the kind of problems that we are able to deal with seems to be of interest in, e.g.,
biological models: indeed, on one hand, equations with 
functional terms in the right-hand side commonly appear in models about cell-adhesion (see, e.g., \cite{HB,LK}); on the other hand, 
the fractional Laplacian is also used to model superdiffusive cells (see, e.g., \cite{ERGPS, PST}).
 
In the above cited papers, variational methods are frequently used to prove
the existence/multiplicity of solutions.
To the best of our knowledge, 
not many papers have been devoted to equations driven by the  fractional 
Laplacian from the point of view of topological methods.
Let us mention here, for instance, the recent papers by 
 Alves,  de Lima and N\'obrega  \cite{Alves2018, Alves2020}:
in these papers, the authors obtained Rabinowitz-type global bifurcation 
results of positive solutions of a parametric fractional 
Laplacian e\-qua\-tion in $\R^n$ using the Leray-Schauder degree.
On the other hand, due to the presence of the non-local functional weights, system \eqref{eq.PBIntro} 
can be viewed as a \emph{nonlinear fractional Kirchhoff-type problem},
even if in a slightly different direction than the one proposed in \cite{CM2020, FV14}.
\medskip

As already pointed out, in this paper we adopt a topological approach based on the classical notion of  fixed point index (see e.g.\,\cite{guolak}) to prove our main existence result, namely Theorem \ref{thm.mainExistence} below; moreover, we prove a non-existence result via an elementary argument.
In some sense, our existence result stems from a pioneering work by Amman \cite{Amann-rev,AmCra} 
and follows a line recently pursued by the authors in the study 
of elliptic PDEs \cite{BiCaInf1, gi-tmna, gi-jepe, gi-BK}.
We point out that our approach permits to consider (possibly nonlinear) functional  BCs:
for example, in Section \ref{sec.examples} we will discuss the solvability of the following problem:
   \begin{equation} \label{eq.PBintro}
    \begin{cases}
    (-\Delta)^{\frac{1}{4}}u_1 = \lambda_1(1-u_1)\,\int_{B_1}e^{u_2}\,\d x & \text{in $B_1$}, \\
    (-\Delta)^{\frac{3}{4}}u_2 = \lambda_2 u_2\cdot\mathrm{osc}_{B_1}(u_1) & \text{in $B_1$}, \\
    u_1\big|_{\R^2\setminus B_1} = \eta_1\cdot u_1(0)u_2(0), \\
    u_2\big|_{\R^2\setminus B_1} = \eta_2\cdot\limsup\limits_{|x|\to\infty}u_1(x),
    \end{cases}
   \end{equation}
in which by $B_1$ we denote the Euclidean ball in $\R^2$ centered at $0$ with radius $1$, and 
by $\mathrm{osc}_{B_1}(\phi)$ we mean he oscillation of the the function $\phi$ on $B_1$.
\medskip

We now briefly describe the structure of our paper. 
In the first part, we per\-form a preliminary study 
of the fractional differential operators which occur in \eqref{eq.PBIntro}:
in Section \ref{sec.preliminaries} we collect some properties and estimates of the solutions of the Dirichlet problem for $(-\Delta)^s$, $s\in (0,1)$, which allow to define a
\emph{Green operator}, denoted by $\mathcal{G}_s$, from $L^\infty(\Omega)$ to $C^{\,0,s}(\R^n)$. These properties are probably known to the experts in the field, nevertheless we include them for the sake of completeness.
We refer the interested reader to the already quoted survey \cite{DNPV}
for a detailed and self-contained introduction to the  fractional Laplacian.
We also discuss the positivity and compactness of the Green operator $\mathcal{G}_s$, thought of as an operator from $L^\infty(\Omega)$ into itself, as well as spectral properties of $\mathcal{G}_s$.
 Roughly speaking, these estimates yield the \emph{a priori} bounds needed to compute the fixed point index in suitable cones of non-negative functions.
We  point out that a challenging feature of our investigation is the choice of the appropriate functional spaces to which the solutions belong. This is discussed in detail in Section~\ref{sec.preliminaries}.
In Section~\ref{sec.existence} we prove our main results,
while the last Section \ref{sec.examples}
contains a 
couple of examples illustrating both our existence and non-existence result.  

\section{Preliminaries and auxiliary results} \label{sec.preliminaries}
 In order to keep the paper as self-contained as possible,
 we collect in this section some definitions and results
 which shall be exploited in the sequel.
\subsection{The $(-\Delta)^s$-Green operator} \label{subsec.Deltasgreen}
 Here we introduce the so-called $(-\Delta)^s$-Gre\-en operator
 and we establish some of its basic properties. 
 Throughout what follows, we take for fixed
 all the notation listed below.
 \begin{itemize}
  \item $\Omega\subseteq\R^n$ is a (non-void) open set
  with smooth boundary and $s\in (0,1)$;
  \item $H^s(\R^n)$ is the usual fractional
  Sobolev space of order $s$, i.e.,
  $$H^s(\R^n) := \Big\{u\in L^2(\R^n):\,\iint_{\R^{2n}}\frac{|u(x)-u(y)|^2}{|x-y|^{n+2s}}\,\d x\,\,d y
  <\infty\Big\};$$
  \item If $U\subseteq\R^n$ is any open set and $\alpha\in (0,1)$
  is fixed, $C^{\,0,\alpha}(\overline{U})$ is the set of the fun\-ctions
  $u:\overline{U}\to\R$ which are H\"older-continuous up to $\overline{U}$, i.e.,
  $$C^{\,0,\alpha}(\overline{U}) := \bigg\{u\in C(\overline{U}):\,
 [u]_{\alpha,U} := \sup_{x\neq y\in U}\frac{|u(x)-u(y)|}{|x-y|^\alpha}<\infty\bigg\}.$$
 In particular, if $u\in C^{\,0,\alpha}(\overline{U})$, we set
 $$\|u\|_{C^{\,0,\alpha}(\overline{U})} := \sup_{U}|u|+[u]_{\alpha,U}.$$
 \item $C_b(\R^n)$ is the Banach
 space of the continuous functions on $\R^n$ which are \emph{globally bounded
 on $\R^n$}, i.e.,
 $C_b(\R^n) = C(\R^n)\cap L^\infty(\R^n)$.
 \end{itemize}
 In order to introduce the $(-\Delta)^s$-Green operator, our starting point
 is the following no\-ta\-ble result due to Ros-Oton and Serra \cite{RosSerra}.
 \begin{theorem}\protect{\cite[Prop.\,1.1]{RosSerra}} \label{thm.RosSerra}
 Let
 $f\in L^\infty(\Omega)$ be fixed. Then, there exists
 a u\-ni\-que \emph{(}weak\emph{)} solution $u_f\in H^s(\R^n)\cap C^{\,0,s}(\R^n)$ of
 the Dirichlet problem
 \begin{equation} \label{eq.pbDeltasHom}
  \begin{cases}
  (-\Delta)^s u = f & \text{in $\Omega$}, \\
  u \equiv 0 & \text{in $\R^n\setminus\Omega$}.
  \end{cases}
 \end{equation}
 This means, precisely, that $u_f\equiv 0$ pointwise in $\R^n\setminus\Omega$ and
 $$\frac{c_{n,s}}{2}\iint_{\R^{2n}}
 \frac{(u_f(x)-u_f(y))(\varphi(x)-\varphi(y))}{|x-y|^{n+2s}}\,\d x\,\d y
 = \int_{\Omega}f\varphi\,\d x\quad
 \forall\,\,\varphi\in C_0^\infty(\Omega).$$
 Furthermore, there exists a constant $C = C(\Omega,s) > 0$ such that
 \begin{equation} \label{eq.estimRosSerra}
  \|{u}_f\|_{C^{\,0,s}(\R^n)}\leq C\,\|f\|_{L^\infty(\Omega)}.
  \end{equation}
 \end{theorem}
 Thanks to Theorem \ref{thm.RosSerra},
 the following definition is well-posed.
 \begin{definition} \label{def.DeltasGreen}
 We define the \emph{$(-\Delta)^s$-Green operator}
 (relative to $\Omega$) as
 \begin{equation} \label{eq.defGreens}
  \mathcal{G}_s: L^\infty(\Omega)\to C^{\,0,s}(\R^n),\qquad
 \mathcal{G}_s(f) := {u}_f,
 \end{equation}
 where $u_f$ is the unique solution of \eqref{eq.pbDeltasHom} (according
 to Theorem \ref{thm.RosSerra}).
 \end{definition}
 With Definition \ref{def.DeltasGreen} at hand, we now proceed
 by proving some `topological' pro\-per\-ties of $\GG_s$.
 We begin with a couple of \emph{continuity/compactness} results.
 \begin{proposition} \label{prop.GGscont}
   The operator $\mathcal{G}_s$ is continuous from
   $L^\infty(\Omega)$ to $C^{\,0,s}(\R^n)$.
  \end{proposition}
  \begin{proof}
   On account of \eqref{eq.estimRosSerra}, for every $f\in L^\infty(\Omega)$ we have
   $$\|\mathcal{G}_s(f)\|_{C^{\,0,s}(\R^n)}
   = \|{u}_f\|_{C^{\,0,s}(\R^n)}\leq C\,\|f\|_{L^\infty(\Omega)},$$
   where $C > 0$ is a constant only depending on $\Omega$ and $s$. From this, since
   $\mathcal{G}_s$ is obviously linear, we immediately infer that
   $\mathcal{G}_s$ is continuous.
  \end{proof}
  \begin{proposition} \label{prop.GGscompact}
  Let $\{f_k\}_{k = 1}^\infty$ be a \emph{bounded} sequence
  in $L^\infty(\Omega)$. Then, there exists
  $v_0\in C(\R^n)$ such that
  $v_0\equiv 0$ on $\R^n\setminus\Omega$ and
  \emph{(}up to a sub-sequence\emph{)}
  \begin{equation} \label{eq.limGGssubseq}
   \lim_{k\to\infty}\mathcal{G}_s(f_k) = v_0 \qquad\text{uniformly in $\R^n$}.
  \end{equation}
In particular, $\mathcal{G}_s$ is compact from $L^\infty(\Omega)$ into 
  $L^\infty(\R^n)$.
  \end{proposition}
  \begin{proof}
   First of all, since $\{f_k\}_{k = 1}^\infty$ is bounded in $L^\infty(\Omega)$,
   it follows from~\eqref{eq.estimRosSerra} that the
   family
   $\{\mathcal{G}_s(f_k)\}_{k = 1}^\infty\subseteq C^{\,0,s}(\R^n)$ is equi-continuous and  
   equi-bounded; as a consequence, since
   $\overline{\Omega}$ is compact, by Arzel\`{a}-Ascoli's theorem
   there exists some function $g\in C(\overline{\Omega})$ such that
   (up to a sub-sequence)
   \begin{equation} \label{eq.GGsfiuniformg}
    \lim_{k\to\infty}\mathcal{G}_s(f_k) = g\qquad\text{uniformly on $\overline{\Omega}$}.
    \end{equation}
   In particular, since $\mathcal{G}_s(f_k)\equiv 0$ on $\R^n\setminus\Omega$
   for every $k\in\N$, we have
   $g\equiv 0$ on $\partial\Omega$.
   Thus, introducing the function
   $v_0:\R^n\to\R$ defined by
   $$v_0(x) := \begin{cases}
    g(x), & \text{if $x\in\Omega$}, \\
    0, & \text{if $x\in\R^n\setminus\Omega$},
    \end{cases}
    $$
   from \eqref{eq.GGsfiuniformg} 
   (and since $\mathcal{G}_s(f_k) \equiv v_0\equiv 0$ on $\R^n\setminus \Omega$)
   we get
   $$\lim_{k\to\infty}\|\mathcal{G}_s(f_k)-v_0\|_{L^\infty(\R^n)}
   = \lim_{k\to\infty}\|\mathcal{G}_s(f_k)-v_0\|_{L^\infty(\Omega)} = 0,$$
   which is precisely the desired \eqref{eq.limGGssubseq}. This ends the proof.
  \end{proof}
  Then, we prove that $\mathcal{G}_s$ is \emph{positive} with respect to the cone
  \begin{equation} \label{eq.Cconepositive}
   \mathcal{C} := \big\{f\in L^\infty(\Omega):\,\text{$f\geq 0$ a.e.\,on $\Omega$}\big\}.
   \end{equation}
  \begin{proposition} \label{cor.PositiveGGs}
    If $\mathcal{C}$ is as in \eqref{eq.Cconepositive}, then
    $\mathcal{G}_s\big(\mathcal{C}\big)\subseteq\mathcal{C}$.
   \end{proposition}
   \begin{proof}
    Let $f\in\mathcal{C}$ be fixed, and let ${u}_f := \mathcal{G}_s(f)$.
    By definition, ${u}_f$ is the unique solution
    of \eqref{eq.pbDeltasHom}
    in $H^s(\R^n)\cap C^{\,0,s}(\R^n)$; thus, since $f\geq 0$ a.e.\,on $\Omega$
    (as $f\in\mathcal{C}$), from the Weak Maximum Principle
    and the continuity of ${u}_f$ we get
    $$\text{${u}_f = \mathcal{G}_s(f)\geq 0$ point-wise in $\R^n$}.$$
    This ends the proof.
   \end{proof}
   We close this section by briefly studying the
   \emph{spec\-trum} of $\mathcal{G}_s$.
   To this end, we first
   recall a result on the eigenvalues
   of $(-\Delta)^s$, which easily follows by combining
   \cite[Prop.\,9]{SerVal} and \cite[Cor.\,1.6]{RosSerra}
   (see also \cite[Prop.\,4]{SerVal2}).
   \begin{theorem} \label{thm.eigenDeltas}
    There exists a countable set $\Lambda\subseteq (0,\infty)$ such that,
    if $\lambda\in\Lambda$, there exists a solution
    $e_{\lambda}\in H^s(\R^n)\cap C^{\,0,s}(\R^n)$ of the eigenvalue problem
    \begin{equation} \label{eq.eigenPbDeltas}
    \begin{cases}
      (-\Delta)^s u = \lambda\,u & \text{in $\Omega$}, \\[0.1cm]
      u \equiv 0 & \text{on $\R^n\setminus\Omega$},
     \end{cases}
     \end{equation}
    This means, precisely, that $e_\lambda\equiv 0$ in $\R^n\setminus\Omega$ and
    $$
     \frac{c_{n,s}}{2}\iint_{\R^{2n}}\frac{\big(e_{\lambda}(x)-e_{\lambda}(y)\big)
     \big(\varphi(x)-\varphi(y)\big)}
     {|x-y|^{n+2s}}\,\d x\,\d y 
      = \lambda\int_{\Omega}e_{\lambda}\,\varphi\,\d x$$
      for every test function $\varphi\in C_0^\infty(\Omega)$.
   \end{theorem}
   Thanks to Theorem \ref{thm.eigenDeltas}, we can prove the
   following proposition.
   \begin{proposition} \label{prop.spectrumGGs}
    Let $r(\mathcal{G}_s)$ denote the spectral radius of $\mathcal{G}_s$,
    thought of as an operator from $L^\infty(\Omega)$ into itself.
    Then, the following facts hold:
    \begin{itemize}
     \item[\emph{(i)}] $r(\mathcal{G}_s) > 0$;
     \item[\emph{(ii)}]
    there exists a function $\phi\in C^{\,0,s}(\R^n)\setminus\{0\}$ such that
    $$\mathcal{G}_s\big(\phi|_{\Omega}\big) = r(\mathcal{G}_s)\,\phi\qquad\text{and}
    \qquad \text{$\phi \equiv 0$ on $\R^n\setminus \Omega$}.$$
    \end{itemize}
   \end{proposition}
   \begin{proof}
    (i)\,\,On account of Theorem \ref{thm.eigenDeltas}, we can find
    a real $\lambda > 0$ and a function 
    $e_\lambda\in C^{\,0,s}(\R^n)\cap H^s(\R^n)$,
    not identically vanishing, such that
    $$
     \begin{cases}
      (-\Delta)^s e_\lambda = \lambda\,e_\lambda & \text{in $\Omega$}, \\[0.1cm]
      e_\lambda \equiv 0 & \text{on $\R^n\setminus\Omega$},
     \end{cases}
    $$
    As a consequence, since $e_\lambda|_{\Omega}\in L^\infty(\Omega)$
    (as $e_\lambda\in C^{\,0,s}(\R^n)$), we have
    $$\mathcal{G}_s\big(e_\lambda|_{\Omega}\big) = \frac{1}{\lambda}\,e_\lambda,$$
    and this proves that $r(\mathcal{G}_s) > 0$, as desired. \medskip
    
    (ii)\,\,We prove the assertion by using the well-known Krein-Rutman theorem.
    To this end we first observe that, if $\mathcal{C}$ is the cone defined
    in \eqref{eq.Cconepositive}, one has: \medskip
    
    (a)\,\,$\mathcal{C}-\mathcal{C}$ is
    dense in $L^\infty(\Omega)$ (actually, $\mathcal{C}-\mathcal{C} = L^\infty(\Omega)$); \medskip
    
    (b)\,\,$\mathcal{G}_s(\mathcal{C})\subseteq\mathcal{C}$ (see 
    Corollary \ref{cor.PositiveGGs}).
    \medskip
    
    \noindent Moreover, from (i) we know that $r(\mathcal{G}_s) > 0$.
    Thus, since Proposition \ref{prop.GGscompact}
    en\-su\-res that $\mathcal{G}_s$ is compact from $L^\infty(\Omega)$ into itself,
    we can invoke Krein-Rutman's theorem, ensuring that $\lambda = r(\mathcal{G}_s)$
    is an eigenvalue of $\mathcal{G}_s$. This means 
    that there exists a function $\phi\in L^\infty(\Omega)$, not identically vanishing,
    such that
    \begin{equation}\label{eq.phieigenfun}
     \mathcal{G}_s(\phi) = r(\mathcal{G}_s)\,\phi.
     \end{equation}
    On the other hand, since $\mathcal{G}_s(\phi) = {u}_\phi\in C^{\,0,s}(\R^n)$
    and vanishes on $\R^n\setminus \Omega$
    (remind that ${u}_\phi$ is the solution
    of \eqref{eq.pbDeltasHom}), 
    from \eqref{eq.phieigenfun} we conclude that
    $$\phi = \frac{1}{r(\mathcal{G}_s)}\,{u}_\phi\in C^{\,0,s}(\R^n)
    \qquad\text{and}\qquad\text{$\phi\equiv 0$ on $\R^n\setminus\Omega$}.$$
    This ends the proof.
   \end{proof}
\subsection{The non-homogeneous Dirichlet problem for $(-\Delta)^s$} \label{subsec.nonhomog}
 Due to its re\-le\-van\-ce in the sequel, we spend a few words about the
 non-homogeneous Dirichlet problem for $(-\Delta)^s$, that is, 
 \begin{equation} \label{eq.pbDirNonHom}
  \begin{cases}
  (-\Delta)^s u = f & \text{in $\Omega$}, \\
  u \equiv \zeta & \text{in $\R^n\setminus\Omega$}.
  \end{cases}
 \end{equation}
 Avoiding to discuss the solvability of \eqref{eq.pbDirNonHom} 
 for general $f$ and $\zeta$, here we limit 
 to observe that, when $\zeta$ is sufficiently regular, problem \eqref{eq.pbDirNonHom} 
 can be reduced
 to \eqref{eq.pbDeltasHom} (with a different $f$).
 In fact, let us suppose that
 $$\zeta\in H^s(\R^n)\cap 
 C_b(\R^n)\cap C^2(\overline{\mathcal{O}}),$$
 where $\mathcal{O}\subseteq\R^n$ is some open neighborhood of $\overline{\Omega}$.
 Then, it is not dif\-ficult to see that
 $(-\Delta)^s\zeta$ can be computed point-wise in $\Omega$, and
 $$(-\Delta)^s\zeta(x)
 = -\frac{c_{n,s}}{2}\int_{\R^n}\frac{\zeta(x+z)+\zeta(x-z)-2\zeta(x)}{|z|^{n+2s}}\,\d z
 \quad (x\in\Omega).$$
 In particular, $(-\Delta)^s\zeta\in L^\infty(\Omega).$
 Moreover, since $\zeta\in H^s(\R^n)$, a standard
 `in\-tegra\-tion-by-parts' argument gives, for every test 
 function $\varphi\in C_0^\infty(\Omega)$,
 \begin{equation} \label{eq.intbypars}
 \int_{\Omega}(-\Delta)^s\zeta\cdot\varphi\,\d x
 = \frac{c_{n,s}}{2}\int_{\R^n\times\R^n}
 \frac{\big(\zeta(x)-\zeta(y)\big)\big(\varphi(x)-\varphi(y)\big)}
     {|x-y|^{n+2s}}\,\d x\,\d y.
  \end{equation}
  On account of these facts, we easily derive the following result.
  \begin{theorem} \label{thm.RosSerranonomog}
   Let $f\in L^\infty(\Omega)$ and let
   $\zeta\in H^s(\R^n)\cap C_b\R^n)\cap C^2(\mathcal{O})$,
   where $\mathcal{O}\subseteq\R^n$ is an open neighborhood of $\overline{\Omega}$.
   Then, the function
   $$u_{f,\,\zeta} := \GG_s\big(f-(-\Delta)^s\zeta\big)+\zeta \in H^s(\R^n)\cap C(\R^n)$$
   is the unique \emph{(}weak\emph{)}
   solution of \eqref{eq.pbDirNonHom}. This means, precisely, that
   $$\frac{c_{n,s}}{2}\iint_{\R^{2n}}
   \frac{(u_{f\,\zeta}(x)-u_{f\,\zeta}(y))(\varphi(x)-\varphi(y))}{|x-y|^{n+2s}}\,\d x\,\d y
   = \int_{\Omega}f\varphi\,\d x\quad
 \forall\,\,\varphi\in C_0^\infty(\Omega),$$
   and $u_{f,\,\zeta} \equiv \zeta$ pointwise on $\R^n\setminus\Omega$.
  \end{theorem}
  \begin{remark} \label{rem.regulufzeta}
   We explicitly notice that, since $f$ and $(-\Delta)^s\zeta$ 
   are in $L^\infty(\Omega)$,
   by the very definition of $\GG_s$ one has
   $v:=\GG_s\big(f-(-\Delta)^s\zeta\big)\in C^{\,0,s}(\R^n)$;
   as a con\-se\-quen\-ce, if $\zeta\in C^{\,0,\alpha}(\R^n)$ for some $\alpha\in (0,1)$, 
   we derive that
   $$u_{f,\,\zeta} = v+\zeta\in C^{\,0,\theta}(\R^n),\qquad
   \text{where $\theta := \min\{s,\alpha\}$}.$$
  \end{remark}
\subsection{The fixed point index} \label{sec.index}
For the sake of completeness, 
we collect in the following proposition some properties of the classical fixed point index that will be crucial in the proof of our existence result;
for more details see, e.g.,~\cite{Amann-rev, guolak}. 

In what follows the closure and the boundary of subsets of a cone $\hat{P}$ are understood to be relative to~$\hat{P}$.
\begin{proposition}\label{propindex}
Let $X$ be a real Banach space and let $\hat{P}\subset X$ be a cone. Let $D$ be an open bounded set of $X$ with $0\in D\cap \hat{P}$ and
$\overline{D\cap \hat{P}}\ne \hat{P}$. 
Assume that $ T:\overline{D\cap \hat{P}}\to \hat{P}$ is a compact operator such that
$x\neq T(x)$ for $x\in \partial (D\cap \hat{P})$. \medskip
Then the fixed point index $i_{\hat{P}}( T, D\cap \hat{P})$ has the following properties:
 
\begin{itemize}

\item[\emph{(i)}] If there exists $e\in \hat{P}\setminus \{0\}$
such that $x\neq T(x)+\sigma e$ for all $x\in \partial (D\cap \hat{P})$ and all
$\sigma>0$, then $i_{\hat{P}}( T, D\cap \hat{P})=0$.

\item[\emph{(ii)}] If $ T(x) \neq \sigma x$ for all $x\in
\partial  (D\cap \hat{P})$ and all $\sigma > 1$, then $i_{\hat{P}}( T, D\cap \hat{P})=1$.

\item[\emph{(iii)}] Let $D^{1}$ be an open bounded subset of $X$ such that
$(\overline{D^{1}\cap \hat{P}})\subset (D\cap \hat{P})$. If 
$i_{\hat{P}}(T, D\cap \hat{P})=1$ and 
$i_{\hat{P}}(T, D^{1}\cap \hat{P})=0$, then $T$ has a fixed point in $(D\cap \hat{P})\setminus
(\overline{D^{1}\cap \hat{P}})$.\\ The same holds if 
$i_{\hat{P}}(T, D\cap \hat{P})=0$ and $i_{\hat{P}}(T, D^{1}\cap \hat{P})=1$.
\end{itemize}
\end{proposition}

   \section{Existence of positive solutions} \label{sec.existence}
   In this section we state and prove our main existence result
   for positive solutions of \eqref{eq.PBIntro}, namely 
   Theorem \ref{thm.mainExistence} below.
   Before doing this, we fix the relevant 
   `struc\-tu\-ral' assumptions which 
   shall be tacitly understood in the sequel.
   \medskip

   As already mentioned in the Introduction, our aim is
   to prove the existence of positive solutions
   for Dirichlet problems 
   of the following form
   
   \begin{equation} \label{eq.MainDirPB}
 \left\{
  \begin{array}{lll}
    (-\Delta)^{s_i}u_i = \lambda_i\,f_i(x,\mathbf{u},\PP_i[\mathbf{u}]) & \text{in $\Omega$}
    & (i = 1,\ldots,m), \\
    u_i(x)= \eta_i\,\zeta_i(x)\,\BB_i[\mathbf{u}],
    & \text{for $x\in \R^n\setminus\Omega$} & (i = 1,\ldots,m).
  \end{array}
\right.
   \end{equation}
   Here, $m\geq 1$ is a fixed integer,
   $\mathbf{u} = (u_1,\ldots,u_m)$, with $u_1,\ldots,u_m:\R^n\to\R$, 
   $s_1,\ldots,s_m\in (0,1)$ and,
   for every fixed $i \in \{1,\ldots,m\}$, we assume that
   \begin{itemize}
   \item[(H0)] $\Omega\subseteq\R^n$ is a bounded open set with smooth boundary;
   \item[(H1)] $f_i$ is a real-valued function defined on $\overline \Omega\times \R^m\times\R$;
   \item[(H2)] $\PP_i,\,\BB_i$ are real-valued operators acting on the space
   \begin{equation} \label{eq.defSpaceX}
    \mathbb{X}  := \big\{u\in C(\R^n;\R^m):\,\sup_{\R^n}|u_i|<\infty\,\,\text{
   for all $i = 1,\ldots,m$}.\big\};
   \end{equation}
   \item[(H3)] $\zeta_i\in H^{s}(\R^n)\cap C_b(\R^n)
   \cap C^2(\mathcal{O})$ and $\zeta_i\geq 0$
   on $\R^n$, where
    \begin{equation} \label{eq.defsmin}
    s = \min_{i = 1,\ldots,m}s_i;
   \end{equation}
   and $\mathcal{O}\subseteq\R^n$ is a suitable open neighborhood of $\overline{\Omega}$;
   \item[(H4)] $\lambda_i,\,\eta_i$ are non-negative parameters.
   \end{itemize}
   In dealing with vector-valued functions $\mathbf{u}=(u_1,\ldots,u_m):\R^n\to\R^m$,
   it is more convenient to use on the spaces $\R^p$ (for $p\in\N$) 
   the \emph{maximum norm}, that is,
   $$\|z\| := \max_{j = 1,\ldots,p}|z_j|\qquad\text{for all $z\in\R^p$}.$$
   Thus, if $\mathbb{X}$ is as in 
   \eqref{eq.defSpaceX} and $\mathbf{u}\in\mathbb{X}$, we define
   $$\|\mathbf{u}\|_\infty := \sup_{x\in\R^n}\|\mathbf{u}(x)\|
   = \max_{i = 1,\ldots,m}\big(\sup_{\R^n}|u_i(x)|\,\big).$$
   Obviously, $(\mathbb{X},\|\cdot\|_\infty)$ is a (real) Banach space. \bigskip
   
   Now, we have already anticipated that we aim to study
   the solvability of problem \eqref{eq.MainDirPB} by means of
   suitable fixed-point techniques; on the other hand,
   since the `non-local' boundary conditions are prescribed on the complementary
   of $\Omega$, one should work in the space $C(\R^n;\R^m)$, which \emph{is not}
   a Banach space.
   
   To overcome this issue, we make the following key observation:
   if
   $\mathbf{u}:\R^n\to\R^m$ is a (continuous)
   function solving \eqref{eq.MainDirPB} point-wise in $\R^n$, then
   $$\text{$u_i(x) = \eta_i\,\zeta_i(x)\,\BB_i[\mathbf{u}]$\quad for all $x$ on $\R^n\setminus\Omega$}
   \qquad
   (i = 1,\ldots,m);$$
   as a consequence, since $\mathbf{u}$ is continuous on $\R^n$
   and $\zeta_i\in C_b(\R^n)$, we deduce that
   $\mathbf{u}\in\mathbb{X}$, and $(\mathbb{X},\|\cdot\|_\infty)$ \emph{is}
   a Banach space.
   We then use the results in Section \ref{sec.preliminaries} to
   define a suitable functional $\mathcal{T}$, acting on 
   the space $\mathbb{X}$,
   allowing us to rephrase problem \eqref{eq.MainDirPB}
   into the fixed-point equation $\mathcal{T}(\mathbf{u}) = \mathbf{u}$ in $\mathbb{X}$.
   \medskip
      
   To begin with, if $i \in \{1,\ldots,m\}$ is fixed, for the sake of simplicity we denote
   $\mathcal{G}^i$ the $(-\Delta)^{s_i}$-Green operator $\mathcal{G}_{s_i}$
   defined in \eqref{eq.defGreens}.
   Let us recall that 
   $$\GG^i:L^\infty(\Omega)\to C^{0,s_i}(\R^n)$$
   and, by definition, 
   $\mathcal{G}^i(\f) = {u}_{\f}$
   is the unique solution  of $(\mathrm{P})_{\f,\,0}$ in 
   $C^{\,0,s_i}(\R^n)$ (with $\f\in L^\infty(\Omega)$). 
   In particular, $\GG^i(\f)\equiv 0$ on $\R^n\setminus \Omega$. As a consequence, if
   $s\in(0,1)$ is as in \eqref{eq.defsmin}
   we derive that
   \begin{equation} \label{eq.KiXinX}
    \mathcal{G}^i\big(L^\infty(\Omega)\big)\subseteq L^\infty(\R^n)
    \cap C^{\,0,s}(\R^n)\qquad
    \text{for all $i = 1,\ldots,m$}.
    \end{equation}
   According to Proposition \ref{prop.spectrumGGs}, 
   we then let $r_i = r(\mathcal{G}^i) > 0$ be the spectral radius
   of $\mathcal{G}^i$, thought of as an operator 
   from $L^\infty(\Omega)$ into itself,
   and we fix once and for all
    a function $\phi_i\in C^{\,0,s}(\R^n)\setminus\{0\}$
    such that (setting $\mu_i := 1/r_i$)
       \begin{equation} \label{def.phii}
    \phi_i = \mu_i\,\mathcal{G}^i\big(\phi_i|_\Omega\big)\qquad\text{and}\qquad
    \text{$\phi_i\equiv 0$ on $\R^n\setminus\Omega$}.
        \end{equation}
   To proceed further, we define
   the \emph{Nemytskii} operator $\mathcal{F}_i$ as
   $$\mathcal{F}_i(\mathbf{u})
   = f_i(\cdot,\mathbf{u}(\cdot),\mathcal{P}_i[\mathbf{u}])\qquad(\mathbf{u}\in\mathbb{X}),$$
   and we assume for a moment that, for any $\bldu \in \mathbb{X}$, we have
   $\mathcal{F}_i(\bldu)\in L^\infty(\Omega)$; we will prove later that this holds under the assumptions of Theorem \ref{thm.mainExistence}. In this case, 
   taking into account \eqref{eq.KiXinX}, we can set
   \begin{equation} \label{eq.defOpI}
	\mathcal{I}(\mathbf{u}) := \Big(\lambda_i\,\mathcal{G}^i
	\big(\mathcal{F}_i(\mathbf{u})\big)\Big)_{i = 1,\ldots,m}\in\mathbb{X}.
   \end{equation}
   Furthermore, if $\gamma_i\in H^s(\R^n)\cap C(\R^n)$ is the unique solution of
   \begin{equation} \label{eq.defgammaiPb}
    \begin{cases}
   (-\Delta)^{s_i}u = 0 & \text{in $\Omega$}, \\[0.1cm]
   u \equiv \zeta_i & \text{in $\R^n\setminus\Omega$}
   \end{cases}
   \end{equation}
   (according to Theorem \ref{thm.RosSerranonomog}), given $\bldu \in \mathbb{X}$ we define
   \begin{equation} \label{eq.defOpD}
    \mathcal{D}(\mathbf{u}) :=
    \Big(\eta_i\,\gamma_i(\cdot)\,\BB_i[\mathbf{u}]\Big)_{i = 1,\ldots,m}.
   \end{equation}
   We explicitly notice that, since $\gamma_i\in C(\R^n)$,
   $\zeta_i\in C_b(\R^n)$ and
   $\gamma_i\equiv\zeta_i$ on $\R^n\setminus\Omega$, we have $\gamma_i\in\mathbb{X}$;
   thus, since $\mathcal{D}(\mathbf{u})$ is a scalar multiple of $\gamma_i$, one has
   $\mathcal{D}(\mathbf{u})\in\mathbb{X}$. \medskip

   Using the operators $\mathcal{I}$ and $\mathcal{D}$
   just introduced,
   we can finally provide the precise
   definition of \emph{solution} of problem
   \eqref{eq.MainDirPB}.
   \begin{definition} \label{def.solutionPB}
    We say that a function $\mathbf{u} = (u_1,\ldots,u_m)\in\mathbb{X}$
    is a~\emph{solution} of problem \eqref{eq.MainDirPB} if
    it satisfies the following properties:
    \begin{itemize}
    \item[(i)] $\mathcal{F}_i(\bldu)\in L^\infty(\Omega)$ for every
    $i = 1,\ldots,m$;
    \item[(ii)] $\bldu = \mathcal{I}(\bldu)+\mathcal{D}(\bldu)$, that is,
    $$u_i = \lambda_i\,\GG^i\big(\mathcal{F}_i(\bldu)\big)+\eta_i\,\gamma_i(\cdot)\,
    \BB_i[\bldu].$$
    \end{itemize}
    If, in addition, $u_i\geq 0$ for all $i = 1,\ldots,m$ and there exists
    some $i_0\in\{1,\ldots,m\}$ such that $u_{i_0}\not\equiv 0$, we say that
    $\mathbf{u}$ is a \emph{non-zero positive solution} of \eqref{eq.MainDirPB}.
   \end{definition}
   \begin{remark} \label{rem.regulsolution}
    On account of Remark \ref{rem.regulufzeta}, if $\zeta_1,\ldots,\zeta_m
    \in C^{\,0,\alpha}(\R^n)$
    for some $\alpha\in (0,1]$, we have that $\gamma_i = u_{\,0,\zeta_i}\in C^{\,0,\theta}(\R^n)$
    for all $i = 1,\ldots,m$,
    where 
    $$\theta = \min\{\alpha,s\}$$
    $s$ being as in \eqref{eq.defsmin}.
    As a consequence,
    we have $\mathcal{D}(\mathbb{X})\subseteq C^{\,0,\theta}(\R^n;\R^m)$ and, by
    \eqref{eq.KiXinX}, any solution of problem
    \eqref{eq.MainDirPB} actually belongs to $C^{\,0,\theta}(\R^n;\R^m)$.
   \end{remark}

   For our existence result, we make use of the classical fixed point index (see Section \ref{sec.index}).
   We will work on the cone 
   \begin{equation} \label{eq.defcon}
   P := \Big\{\mathbf{u}\in \mathbb{X}:\,\text{$u_i\geq 0$ on $\R^n$ for every $i = 1,\ldots,m$}\Big\}.
   \end{equation}
Given a finite sequence
$\varrho = \{\rho_i\}_{i=1}^m\subseteq (0,+\infty)$, we 
define
\begin{equation} \label{eq.defiIvarrho}
I(\varrho)=\prod_{i = 1}^m\,[0,\rho_k] \subseteq \R^m
\end{equation}
   and set 
   \begin{equation} \label{eq.defconrho}
   P(\varrho) := \Big\{\mathbf{u}\in \mathbb{X}:\,
   \text{$\mathbf{u} (x) \in I(\varrho)$ for all $x \in \R^n$}\Big\}.
   \end{equation}
  
 \begin{theorem} \label{thm.mainExistence}
Let the assumptions \emph{(H0)}-to-\emph{(H4)} be in force. Moreover, 
let us sup\-po\-se that there exists
a finite sequence
$\varrho = \{\rho_i\}_{i = 1}^m\subseteq(0,\infty)$ satisfying the following hypotheses:
\begin{itemize}
\item[\emph{(a)}] For every $i = 1,\ldots,m$, one has that
\begin{itemize}
 \item[$\mathrm{(a)}_1$] $\PP_i\big|_{P(\varrho)}$ is continuous, and there exist 
 $\underline\omega_{i,\varrho}, \overline\omega_{i,\varrho} \in \R$ such that
 $$
 \underline\omega_{i,\varrho} \leq \mathcal{P}_i[\mathbf{u}] \leq \overline\omega_{i,\varrho}, \quad \text{for every $\bldu \in P(\varrho)$};
 $$ 
 \item[$\mathrm{(a)}_2$] $f_i$ is continuous and non-negative on 
 $\overline \Omega\times I(\varrho)\times
 [\underline\omega_{i,\varrho}, \overline\omega_{i,\varrho}]$; \vspace*{0.05cm}
 \item[$\mathrm{(a)}_3$] $\BB_i\big|_{P(\varrho)}$ is continuous, non-negative, 
 and bounded;
\end{itemize}

\item[\emph{(b)}] There exist $\delta \in (0,+\infty)$, $i_0\in \{1,2,\ldots,{m}\}$ 
and $\rho_0 \in (0,\min_{i}{\rho_i}\big)$ 
such that, if $\varrho_0$ denotes the finite sequence 
$\varrho_0 := \{\rho_0\}_{i = 1}^m$, we have
\begin{equation} \label{eq.mainestimk0}
 f_{i_0}(x,\bldz,\omega)\ge \delta z_{i_0}\quad \text{for every $(x,\bldz,\omega)\in \Pi_0$},
\end{equation}
where $\Pi_{0}:=\overline \Omega \times I(\varrho_0)
\times [\underline\omega_{0}, \overline\omega_{0}]$ and
$$\underline\omega_{0}
 := \inf_{\bldu\in P(\varrho_0)}\PP_{i_0}[\mathbf{u}],\qquad
 \overline\omega_{0}:=\sup_{\bldu\in P(\varrho_0)}\PP_{i_0}[\mathbf{u}].$$

\item[\emph{(c)}] Setting, for every $i = 1,\ldots,m$,
\begin{equation} \label{eq.defiMkHk}
\begin{split}
 & M_{i}:=\max \Big\{f_{i}(x,\bldz,\omega): (x,\bldz,\omega)\in \overline \Omega\times I(\varrho)\times [\underline\omega_{i,\varrho}, \overline\omega_{i,\varrho}]\Big\}\quad \text{and}
\\
& B_i:=\sup_{\bldu\in {P(\varrho)}}\BB_i[\bldu]
\end{split}
\end{equation}
the following inequalities are satisfied: 
\begin{itemize}
 \item[$\mathrm{(c)}_1$] ${\mu_{i_{0}}}\leq \delta\lambda_{i_0}$; \vspace*{0.08cm}
 
 \item[$\mathrm{(c)}_2$] 
 $\lambda_i\,M_i\,\|\GG^i(\hat{1})\|_{\infty} + 
 \eta_i\,B_i\|\gamma_i\|_{\infty} \leq \rho_i$.
 \end{itemize}
\end{itemize}
Then the system~\eqref{eq.MainDirPB} has a non-zero positive solution $\bldu\in\mathbb{X}$ such that 
\begin{equation} \label{eq.estimnonzero}
\|\bldu\|_{\infty}\geq \rho_0\qquad
\text{and} \qquad \text{$\|u_i\|_{\infty}\leq \rho_i$ for every $i=1,\ldots,m$}.
\end{equation}
   \end{theorem}

\begin{proof}
As a preliminary fact, we explicitly observe that 
$\mathcal{T} := \mathcal{I}+\mathcal{D}$
is \emph{a well-defined operator} from $P(\varrho)$ to $\mathbb{X}$.
For this purpose, we show that
\begin{equation} \label{eq.FFiLinf}
 \mathcal{F}_i\big(P(\varrho)\big)\subseteq L^\infty(\Omega)
 \qquad (\text{for every $i = 1,\ldots,m$}).
\end{equation}
In fact, given $\bldu \in P(\varrho)$, by assumption $\mathrm{(a)}_1$ we have, for every $i = 1,\ldots,m$,
$$(x,\bldu(x),\PP_i[\bldu])\in \overline{\Omega}\times I(\varrho)
\times[\underline\omega_{i,\varrho}, \overline\omega_{i,\varrho}]\quad 
\text{for all $x\in\overline{\Omega}$};$$
as a consequence, from assumption $\mathrm{(a)}_2$ we readily
derive \eqref{eq.FFiLinf}. Notice that $\mathcal{T}$ maps $P(\varrho)$ into
$\mathbb{X}$ in view of \eqref{eq.KiXinX} and the very definition of $\mathcal{D}$.
\medskip 

We now show that the following assertions hold:
\begin{itemize}
 \item[(i)] $\mathcal{T}$ maps $P(\varrho)$ into $P\subseteq\mathbb{X}$;
 \item[(ii)] $\mathcal{T}:P(\varrho)\to P$ is compact.
\end{itemize}
To prove (i), let $\bldu\in P(\varrho)$ and let
$i\in\{1,\ldots,m\}$ be fixed. 
We first observe that, since $\bldu\in P(\varrho)$, by
$(\mathrm{a})_3$ we have $\BB_i[\bldu]\geq 0$; moreover, since
$\zeta_i\geq 0$ on $\R^n$ (see assumption (H3)), from
 the Weak Maximum Principle we derive that
 $\gamma_i\geq 0$.
Thus, the fact that $\eta_i$ is nonnegative implies that
$$\mathcal{D}(\bldu)_i = \eta_i\,\gamma_i(\cdot)\, \BB_i[\bldu]\geq 0\quad \text{on $\R^n$}.$$
On the other hand, since $\bldu\in P(\varrho)$, by assumption
$(\mathrm{a})_2$ we also have
$$\mathcal{F}_i(\mathbf{u})
   = f_i(\cdot,\mathbf{u}(\cdot),\mathcal{P}_i[\mathbf{u}])\geq 0;$$
as a consequence, from Corollary \ref{cor.PositiveGGs} we infer that
$\mathcal{G}^i \big(\mathcal{F}_i(\mathbf{u})\big) \geq 0$ on $\R^n$,
and thus, as $\lambda_i \geq 0$, we get
\begin{equation*}
\mathcal{I}(\bldu)_i = \lambda_i\GG^i(\mathcal{F}_i(\bldu))\geq 0\quad \text{on $\R^n$}.
 \end{equation*}
 Gathering together these facts, and bearing in mind
 the very definition of $\mathcal{T}$, we conclude that 
 $\mathcal{T}(P(\varrho))\subseteq P$, as claimed. \vspace*{0.05cm}

 We now prove assertion (ii). To this end, we fix
 $i\in\{1,\ldots,m\}$ and we observe that, in view assumption $\mathrm{(a)_2}$,
 we have that
 $\mathcal{F}_i(\mathbf{u}): P(\varrho) \to L^\infty(\Omega)$ is continuous; moreover,
by Proposition \ref{prop.GGscompact} one also has that
 $$\mathcal{G}^i: L^\infty(\Omega) \to L^\infty(\R^n)$$
 is linear and compact.
As a consequence, we deduce that
 $$\mathcal{I} = (\lambda_i\,\GG^i\circ\mathcal{F}_i)_{i = 1,\ldots,m}:
 P(\varrho)\to L^\infty(\R^n;\R^m)$$
 is compact. On the other hand, since $\mathcal{D}$ is bounded
 on $P(\varrho)$ (as the same is true of $\BB_1,\ldots,\BB_m$,
 see assumption $\mathrm{(a)}_3$) and since, by definition,
 $$\mathcal{D}(P(\varrho))\subseteq\big\{t\boldsymbol{\gamma}:\,t\in\R\big\}
 \subseteq L^\infty(\R^n;\R^m)\qquad
 (\text{where $\boldsymbol{\gamma}:=(\gamma_1,\ldots,\gamma_m)$}),$$
 we immediately derive that $\mathcal{T} = \mathcal{I}+\mathcal{D}$ is compact
 from $P(\varrho)$ into $L^\infty(\R^n;\R^m)$.
 From this, since 
 $\mathcal{T}(P(\varrho))\subseteq P$ (by assertion (i)) and since
 $P$ is a closed subspace of $L^\infty(\R^n;\R^m)$, we conclude that 
 $\mathcal{T}$ is compact from $P(\varrho)$ into $P$.
 \medskip

 To proceed further, we define
 $$P_0=\Big\{\bldu\in \mathbb{X}:\, 
 \bldu(x) \in I(\varrho_0) \mbox{ for all  } x \in\R^n \Big\}\subseteq 
 P(\varrho)\subseteq 
 P,$$ 
 where $\varrho_0$ is as in assumption (b). Moreover, we consider the open sets
 \begin{align*}
  & D := \big\{\bldu\in\mathbb{X}:\,\text{$\|u_i\|_\infty<\rho_i$ for every
  $i = 1,\ldots,m$}\big\} \qquad\text{and} \\[0.1cm]
  & \qquad D^1 := \big\{\bldu\in\mathbb{X}:\,
  \text{$\|u_i\|_\infty < \rho_0$ for every $i = 1,\ldots,m$}\big\}.
 \end{align*}
 We explicitly observe that, if $D,\,D^1$ are as above, we have
 $\de(D\cap P) = \de P(\varrho)$ and
 $\de (D^1\cap P) = \de P_0$,
 where both the boundaries
 are  relative to $P$.\vspace*{0.05cm}
 
 Now, if
 the operator $\mathcal{T}$ has a fixed point  
 $\bldu_0\in \de P(\varrho)\cup\de P_0$ , then $\bldu_0$ is a solution
 of problem
 \eqref{eq.MainDirPB} satisfying \eqref{eq.estimnonzero}, and the theorem is proved.
 If, instead, $\mathcal{T}$
 is fixed-point free on $\de P(\varrho)\cup\de P_0$, both the fixed-point indices
 $$i_{{P}}(\mathcal{T}, D\cap P) \qquad\text{and}
 \qquad i_{{P}}(\mathcal{T}, D^1\cap P)$$
 are well-defined. Assuming this last possibility, we prove the following. \medskip
 
 \textsc{Claim 1.} We claim that 
 \begin{equation} \label{eq.toproveIndexone}
  i_{{P}}(\mathcal{T}, D\cap P) = 1.
  \end{equation}
 According to Proposition \ref{propindex}-(ii), to prove
 \eqref{eq.toproveIndexone} it suffices to show that
 \begin{equation} \label{eq.toprovenofixedpointBDone}
  \AA(\bldu) \neq \sigma\,\bldu \qquad\text{for every $\bldu\in\de P(\varrho)$
  and every $\sigma > 1$},
 \end{equation}
 To
 establish \eqref{eq.toprovenofixedpointBDone} we argue by contradiction,
 and we suppose that there exist 
 a function $\bldu\in \de P(\varrho)$ and a real $\sigma >1$ such that 
 $$\sigma\bldu= \mathcal{T}(\bldu).$$ 
 Since $\bldu\in \de P(\varrho)$, 
there exists an index $i\in \{1,\ldots,m\}$ such 
that $\| u_i\|_{\infty} = \rho_i$.
By
 assumption $\mathrm{(a)_1}$ and
 \eqref{eq.defiMkHk}, we have
 \begin{equation} \label{eq.touseWMP}
  0 \leq \mathcal{F}_{i}(\bldu)(x) = 
  f_i(x,\bldu(x),\mathcal{P}_i[\mathbf{u}]) \leq M_i \qquad\text{for all $x\in\overline \Omega$};
 \end{equation}
so that
\begin{equation} \label{idx1in}
\begin{split}
 \sigma u_i (x) & =
 \lambda_i\,\GG^i\big(\mathcal{F}_i(\bldu)\big)(x) +
  \eta_i\,\gamma_i(x)\,\BB_i[\bldu] \\
  & \leq \lambda_i\,\GG^i\big(M_i\hat{1}\big)(x) +
  \eta_i\,\gamma_i(x)\,\BB_i[\bldu] \\
&  \leq \big\|\lambda_i\,\GG^i\big(M_i\hat{1}\big)\big\|_{\infty}
+ \big\|\eta_i\,B_i\,\gamma_i\big\|_{\infty}  \\
& 
= \lambda_i\,M_i\,\big\|\GG^i(\hat{1})\|_{\infty} 
+ \eta_i\,B_i\,\|\gamma_i\|_{\infty} 
 \leq \rho_i
 \end{split}
\end{equation}
the last inequality following by assumption $\mathrm{(c)}_2$.
As a consequence, by taking the supremum for $x\in \overline \Omega$ in \eqref{idx1in},
as $\bldu\in \de P(\varrho)\subseteq P(\varrho)$, we get 
$$\sup_{x\in \overline \Omega}|\sigma\,u_i(x)| \leq \sigma\,\rho_i \leq \rho_i,$$ 
which is clearly a contradiction (since $\sigma > 1$), and the claim is proved. \medskip

\textsc{Claim 2.} We claim that
 \begin{equation} \label{eq.toproveIndexzero}
  i_{{P}}(\mathcal{T}, D^1\cap P) = 0.
  \end{equation}
 According to Proposition \ref{propindex}-(i), to prove
 \eqref{eq.toproveIndexzero} it suffices to show that
 there exists a suitable function $e\in P\setminus\{0\}$ satisfying the property
 \begin{equation} \label{eq.toprovenofixedpointBDzero}
  \mathcal{T}(\bldu) + \sigma e \neq \bldu \qquad\text{for every $\bldu\in\de P_0$
  and every $\sigma > 0$}.
 \end{equation}
 To
 establish \eqref{eq.toprovenofixedpointBDzero}, we let
 $e := (\phi_1,\ldots,\phi_m)$,
where each component  $\phi_1,\ldots,\phi_m$
 is as in \eqref{def.phii}, and we argue by contradiction: we thus
 suppose that there exist 
 $\bldu\in \partial P_0$ and $\sigma  >0$ such that
$$
\bldu= \mathcal{T}(\bldu) + \sigma e. 
$$

Let $i_0$ be as in assumption (b). Since $\mathcal{T}(\bldu)\in P$, we have
$$u_{i_0} = \mathcal{T}(\bldu)_{i_0} 
+ \sigma\,\phi_{i_0}\geq \sigma  \phi_{i_0} \qquad\text{on $\overline \Omega$}.$$
Furthermore, again by assumption (b), for
every $x\in\overline{\Omega}$ we get
\begin{equation} \label{eq.estimFk0varphi}
 \mathcal{F}_{i_0}(\bldu)(x) = f_{i_0}(x,\bldu(x),\mathcal{P}_{i_0}[\mathbf{u}])
  \geq \delta u_{i_0}(x)\geq \delta\sigma\phi_{i_0}(x).
\end{equation}
Gathering together all these facts, for every $x\in \overline \Omega$ we have
\begin{equation*}
\begin{split}
 u_{i_0}(x) & =
  \lambda_{i_0} \GG^{i_0}\big(\mathcal{F}_{i_0}(\bldu)\big)(x)
  +\eta_{i_0}\,\gamma_{i_0} (x)\,\BB_{i_0}[\bldu]  
  + \sigma\,\phi_{i_0} (x) \\
 & \geq 
\lambda_{i_0}\,\GG^{i_0}(\delta\sigma\phi_{i_0})(x) +
\sigma\phi_{i_0} (x) \\
& =  \frac{\delta\lambda_{i_0}}{\mu_{i_0}}\cdot\sigma\phi_{i_0}(x)
 + \sigma\phi_{i_0} (x)  \geq 2\sigma  \phi_{i_0} (x)
\end{split}
\end{equation*}
the last inequality following by assumption $\mathrm{(c)}_1$.

By iterating the above argument, for every $x\in \overline \Omega$ we get
$$
u_{i_0} (x)\geq p\sigma\phi_{i_0}(x) \qquad \text{for every $p\in\mathbb{N}$},
$$
a contradiction since $u_{i_0}$ is bounded. \medskip

\noindent We are now ready to conclude the proof of the theorem: in fact,
by combining Claims 1 and 2 and Proposition
\ref{propindex}-(iii), we infer the existence of a fixed point 
$$\bldu_0\in \big(D\cap P\big) \setminus P_0$$ 
of $\mathcal{T}$; thus, $\bldu_0$ is a solution of
\eqref{eq.MainDirPB} satisfying \eqref{eq.estimnonzero}.
\end{proof}

An elementary argument yields the following non-existence result.

\begin{theorem}\label{thmnonex} 
Let the assumptions \emph{(H0)}-to-\emph{(H4)} be in force. Moreover, 
let us suppose that there exists
a finite sequence
$\varrho = \{\rho_i\}_{i = 1}^m\subseteq(0,\infty)$
such that, for every $i = 1,\ldots,m$, the following conditions hold:
\begin{itemize}
\item[\emph{(a)}] 
there exist $\underline{\omega}_{i,\varrho},\,\overline{\omega}_{i,\varrho}\in\R$ such that
\begin{equation}
 \underline{\omega}_{i,\varrho} \leq \PP_i[\bldu]\leq \overline{\omega}_{i,\varrho}
 \quad\text{for every $\bldu\in P(\varrho)$};
\end{equation} 
\item[\emph{(b)}] there exist $\tau_i \in (0,+\infty)$
such that
$$f_{i}(x,\bldz,\omega)\le \tau_i z_i\quad \text{for every $(x,\bldz,\omega)\in 
\overline \Omega\times I(\varrho)\times [
\underline{\omega}_{i,\varrho},\overline{\omega}_{i,\varrho}]$},$$ 
\item[\emph{(c)}] there exist $\xi_i \in (0,+\infty)$
such that
$$\big|\BB_{i}[\bldu]\big| \le \xi_i \cdot \|\bldu\|_\infty, \quad \text{for every $u\in 
P(\varrho)$},$$ 
\item[\emph{(d)}] the following inequality holds: 
\begin{equation} \label{eq.estimleq1absurd}
 \lambda_i\tau_i\,\|\GG^i(\hat{1})\|_{\infty} + 
 \eta_i\,\xi_i\|\gamma_i\|_{\infty} <1.
 \end{equation}
\end{itemize}

Then the system~\eqref{eq.MainDirPB} has at most the zero solution in $P(\varrho)$.
\end{theorem}

\begin{proof}
 By contradiction, assume that \eqref{eq.MainDirPB}
 has a solution $\bldu\in P(\varrho)\setminus\{0\}$, that is, 
 for every $i = 1,\ldots,m$ we have (see Definition \ref{def.solutionPB}):
 $$\mathcal{F}_i(\bldu)\in L^\infty(\Omega)\quad\text{and}\quad
 u_i = \lambda_i\,\GG^i\big(\mathcal{F}_i(\bldu)\big)+
  \eta_i\,\gamma_i(\cdot)\BB_i[\bldu].$$
 Setting $\rho := \|\bldu\|_\infty > 0$, we let 
 $j\in \{1,2,\ldots,m\}$ be such that  
 \begin{equation} \label{eq.choicejabsurd}
 \|u_{j}\|_{\infty}=\rho.
 \end{equation}
 In view of assumptions (a)-(b), for every $x\in\overline \Omega$ we then have 
 \begin{equation} \label{eq.touseabsurd}
  \mathcal{F}_{j}(\bldu)(x) = 
  f_j(x,\bldu(x),\mathcal{P}_j[\mathbf{u}]) \leq \tau_j u_j(x) \leq \tau_j \rho,
 \end{equation}
 and thus (see Corollary \ref{cor.PositiveGGs}, and recalling that
 $\mathcal{F}_j(\bldu)\in L^\infty(\Omega)$)
 $$\GG^j\big(\tau_j\rho\cdot\hat{1}-\mathcal{F}_j(\bldu)\big)\geq 0\,\,
 \Longleftrightarrow\,\,\GG^{j}(\mathcal{F}_j(\bldu))
 \leq \tau_j\rho\,\GG^j(\hat{1})\qquad \text{on $\R^n$}.$$
 As a consequence, we obtain
\begin{equation} \label{nonext}
\begin{split}
u_{j}(x) & =
 \lambda_j\,\GG^j\big(\mathcal{F}_j(\bldu)\big)(x) +
  \eta_j\,\gamma_j(x)\,\BB_j[\bldu] \\ 
  & \leq \lambda_j\tau_j\rho\,\GG^j(\hat{1})(x) +
  \eta_j\,\gamma_j(x)\,\BB_j[\bldu] \\
  & (\text{by assumption (c) and \eqref{eq.choicejabsurd}}) \\
  & \leq \|\lambda_j\tau_j\rho\,\GG^j(\hat{1})\|_{\infty}
+ \big\|\eta_j\,\xi_j \rho\,\gamma_j\big\|_{\infty}  \\
& = \big(\lambda_j\,\tau_j\,\big\|\GG^j(\hat{1})\|_{\infty} 
+ \eta_j\,\xi_j\,\|\gamma_j\|_{\infty}\big) \rho.
\end{split}
\end{equation}
By taking the supremum in~\eqref{nonext} for $x\in \overline \Omega$, 
from \eqref{eq.estimleq1absurd} and \eqref{eq.choicejabsurd} 
we get
$$\rho = \sup_{x\in\overline \Omega}u_j(x) \leq 
\left(\lambda_j\,\tau_j\,\big\|\GG^j(\hat{1})\|_{\infty} 
+ \eta_j\,\xi_j\,\|\gamma_j\|_{\infty}\right)\rho < \rho,$$
and this is clearly a contradiction. Thus, we conclude that
problem \eqref{eq.MainDirPB}
cannot have nonzero solutions  in $P(\varrho)$, and the proof
is complete.
\end{proof}

\section{Examples} \label{sec.examples}
  In this last section we present a couple of concrete examples
  illustrating the applicability of 
  Theorems \ref{thm.mainExistence} and \ref{thmnonex}. 
  Before proceeding we remind the following result,
  which shall play a key role in our computations.
  \begin{lemma} \label{lem.Solutiondir}
   Let $r > 0$ be fixed, and let $B_r\subseteq\R^n$ be the Euclidean
   ball centered at $0$ with radius $r$. Moreover,
   let $s\in (0,1)$. Then, the unique solution $v_s$ of
   $$\begin{cases}
      (-\Delta)^s v = 1 & \text{in $B_r$}, \\
      v\equiv 0 & \text{on $\R^n\setminus B_r$}
      \end{cases}$$
      has the following explicit expression
      \begin{equation} \label{eq.explicitvsLemma}
      v_s(x)
      = \frac{2^{-2s}\Gamma(n/2)}{\Gamma(\frac{n+2s}{2})\,\Gamma(1+s)}\,(r^2-\|x\|^2)^s_+,
      \end{equation}
      where $\|\cdot\|$ stands for
      the usual Euclidean norm and
      $$\Gamma(\alpha) = \int_{\R}x^{\alpha-1}\,e^{-x}\,\d x\qquad(\alpha > 0).$$
  \end{lemma}
  \noindent For a proof of Lemma \ref{lem.Solutiondir} we refer, e.g., to \cite{BGR, Get}.  
  \begin{example} \label{exm.existence}
   In Euclidean space $\R^2$, let us consider the following BVP
   \begin{equation} \label{eq.PBexmesit}
    \begin{cases}
    (-\Delta)^{\frac{1}{4}}u_1 = \lambda_1(1-u_1)\,\int_{B_1}e^{u_2}\,\d x & \text{in $B_1$}, \\
    (-\Delta)^{\frac{3}{4}}u_2 = \lambda_2 u_2\cdot\mathrm{osc}_{B_1}(u_1) & \text{in $B_1$}, \\
    u_1\big|_{\R^2\setminus B_1} = \eta_1\cdot u_1(0)u_2(0), \\
    u_2\big|_{\R^2\setminus B_1} = \eta_2\cdot\limsup\limits_{|x|\to\infty}u_1(x),
    \end{cases}
   \end{equation}
   where $B_1$ is the Euclidean ball centered at $0$ with radius $1$, and
   \begin{equation} \label{eq.defOsciexm}
   \mathrm{osc}_{B_1}(\phi) := \sup_{B_1}(\phi)-\inf_{B_1}(\phi)\qquad
   (\text{for all $\phi\in\mathbb{X}$}).
   \end{equation}
   Clearly, problem \eqref{eq.PBexmesit} is of the form \eqref{eq.MainDirPB}, with
   \begin{enumerate}
    \item $\Omega = B_1$, $m = 2$, $s_1 = 1/4$ and $s_2 = 3/4$;
    \item $f_1:\overline{B}_1\times\R^2\times\R,\quad 
    f_1(x,\bldz,w) := (1-z_1)\,w$;
    \item $f_2:\overline{B}_1\times\R^2\times\R,\quad 
    f_2(x,\bldz,w) := z_2\,w$;
    \item $\PP_1:\mathbb{X}\to\R,\quad \PP_1[\bldu] := \int_{B_1}e^{u_2}\,\d x$;
    \item  $\PP_2: \mathbb{X}\to\R,\quad \PP_2[\bldu] := \mathrm{osc}_{B_1}(u_1)$;
    \item  $\BB_1:\mathbb{X}\to\R,\quad \BB_1[\bldu] := u_1(0)u_2(0)$;
    \item  $\BB_2:\mathbb{X}\to\R,\quad \BB_1[\bldu] := \limsup_{|x|\to\infty}u_1(x)$;
    \item  $\zeta_1\equiv\zeta_2\equiv 1$.
   \end{enumerate}
   Moreover, it is straightforward to recognize that
   all the `structural' assumptions (H0)-to-(H4) listed
   at the beginning of Section \ref{sec.existence} are fulfilled.
  We now turn to prove that, in this case,
   also assumptions (a)-to-(c) of Theorem \ref{thm.mainExistence}
   are satisfied
   for a suitable choice of the nonnegative parameters $\lambda_1,\lambda_2,\eta_1,\eta_2$. \vspace*{0.05cm}
   
   To this end, we consider the (finite) sequence $\varrho$ defined as follows:
   \begin{equation} \label{eq.defrhoExmexist}
    \varrho = \{\rho_1,\rho_2\},\qquad
    \text{where $\rho_1 := \frac{1}{2}$ and $\rho_2 := 1$}.
   \end{equation}
   According to this choice of $\varrho$, we have
   (see \eqref{eq.defiIvarrho}-\eqref{eq.defconrho})
   $$I(\varrho) = [0,1/2]\times[0,1],\qquad
   P(\varrho) = \big\{u\in\mathbb{X}:\,\text{$\bldu(x)\in I(\varrho)$ for
   all $x\in\R^2$}\big\}.$$
   
   \noindent\textbf{Assumption (a).} 
   First of all, it is easy to see that $\PP_1,\PP_2$ are continuous when restricted to
   the set $P(\varrho)\subseteq\mathbb{X}$; moreover,
   for every $\bldu = (u_1,u_2)\in\mathbb{X}$ we have
   \begin{align*}
    & \pi\leq\PP_1[\bldu] = \int_{B_1}e^{u_2}\,\d x\leq \pi\cdot e \quad\text{and}
    \quad
    0\leq \PP_2[\bldu] = \mathrm{osc}(u_1) \leq \frac{1}{2},
   \end{align*}
   so that
   assumption $\mathrm{(a)_1}$ is satisfied with the choices
   \begin{equation} \label{eq.choiceomega}
   \underline{\omega}_{1,\varrho} := \pi,\quad
   \overline{\omega}_{1,\varrho}
    := \pi\cdot e,\quad 
   \underline{\omega}_{2,\varrho} := 0,\quad \overline{\omega}_{2,\varrho}
    := 1/2.
    \end{equation}
    As regards assumption $\mathrm{(a)_2}$, we first notice
    that $f_1,f_2\in C(\overline{B}_1\times\R^2\times\R)$;
    moreover, by taking into account \eqref{eq.defrhoExmexist} and
    \eqref{eq.choiceomega}, we get
    \begin{align*}
     & \text{$f_1(x,\bldz,w) = (1-z_1)\,w \geq \frac{\pi}{2} > 0$\,\,
     on $\overline{B_1}\times I(\varrho)\times [\pi,\pi\cdot e]$}
     \qquad\text{and} \\[0.05cm]
     & \qquad \text{$f_2(x,\bldz,w) = z_2\,w \geq 0$\,\,
     on $\overline{B_1}\times I(\varrho)\times [0,1/2]$},
     \end{align*}
     so that
     $\mathrm{(a)_2}$ is fulfilled.
     Finally, as regards assumption $\mathrm{(a)_3}$, it is not difficult
     to check that $\BB_1,\BB_2$ are continuous and non-negative
     when restricted to $P(\varrho)$; moreover, since
     for every $\bldu = (u_1,u_2)\in P(\varrho)$ we have
     \begin{align*}
     |\BB_1[\bldu]| = |u_1(0)u_2(0)| \leq \frac{1}{2}\quad
     \text{and}\quad
     |\BB_2[\bldu]| = |\limsup_{|x|\to\infty}u_1(x)|\leq \frac{1}{2}
     \end{align*}
     we conclude that $\BB_1,\BB_2$ are bounded on $P(\varrho)$.
     \vspace*{0.05cm}
     
     \noindent\textbf{Assumption (b).} 
      First of all, if $\rho_0\in(0,1/2)$ is arbitrarily fixed, we have
      $$\underline{\omega}_0 := \inf_{\bldu \in P(\varrho_0)}\PP_1[\bldu]
      = \pi,$$
      where $\varrho_0 := \{\rho_0, \rho_0\}$;
      moreover, for every $(x,\bldz)\in\overline{B}_1\times I(\varrho_0)$
      and every $w\geq \pi$, one has
      \begin{align*}
       f_1(x,\bldz, w) = (1-z_1)\,w \geq \frac{\pi}{2}.
      \end{align*}
      Gathering together these facts, we easily
      conclude that \eqref{eq.mainestimk0} holds
      \emph{for every choice of $\delta > 0$}. In fact, given
      any such $\delta$, we define
      \begin{equation} \label{eq.choicerho0}
      \rho_0 = \rho_0(\delta) := \min\Big\{\frac{1}{4},\frac{\pi}{2\delta}\Big\}
      \in (0,1/2);
      \end{equation}
      then, for every $(x,\bldz)\in \overline{B}_1\times I(\varrho_0)$
      and every $w\geq \pi = \underline{\omega}_0$, we get
      $$f_1(x,\bldz,w)\geq \frac{\pi}{2\delta}\cdot\delta
      \geq \delta z_1,$$
      and thus assumption (b) is satisfied with $i_0 = 1$
      (and for every $\delta > 0$).
      \vspace*{0.05cm}
      
      \noindent\textbf{Assumption (c).} We start by
      computing the constants appearing in 
      \eqref{eq.defiMkHk}. First of all, using
      \eqref{eq.defrhoExmexist}, \eqref{eq.choiceomega} and
      the definition of $f_1,f_2$ we get
      \begin{equation} \label{eq.explicitM1M2}
      \begin{split}
       & M_1 = \sup\big\{f_1(x,\bldz,w):\,
       (x,\bldz,w)\in\overline{B}_1\times I(\varrho)\times [\underline{\omega}_{1,\varrho},
       \overline{\omega}_{1,\varrho}]\big\} = \pi\cdot e\quad\text{and}
       \\
       &  
       M_2 = \sup\big\{f_2(x,\bldz,w):\,
       (x,\bldz,w)\in\overline{B}_1\times I(\varrho)\times [\underline{\omega}_{2,\varrho},
       \overline{\omega}_{2,\varrho}]\big\} = \frac{1}{2}.
      \end{split}
      \end{equation}
      Moreover, again by 
      \eqref{eq.defrhoExmexist}, \eqref{eq.choiceomega} and
      the definition of $\BB_1,\BB_2$ we get
      \begin{equation} \label{eq.explicitB1B2}
       B_1 = \sup_{\bldu\in P(\varrho)}
       \BB_1[\bldu] = \frac{1}{2}\quad\text{and}\quad
       B_2 = \sup_{\bldu\in P(\varrho)}
       \BB_2[\bldu]=\frac{1}{2}
      \end{equation}
      We then turn our attention to the functions $\GG^i(\hat{1})
      = \GG_{s_i}(\hat{1})$
      and $\gamma_i$ (for $i = 1,2$).
      To begin with, according to the very definition
      of $(-\Delta)^{s_i}$-Green operator, we know that
      $\GG^i(\hat{1})$ is the unique solution in $C^{0,s_i}(\R^2)$ of
      $$
       \begin{cases}
      (-\Delta)^{s_i} v = 1 & \text{in $B_1$}, \\
      v\equiv 0 & \text{on $\R^2\setminus B_1$}.
      \end{cases}$$
      On the other hand, thanks to Lemma \ref{lem.Solutiondir} we
      can write \emph{the explicit expression}
      of $\GG^i(\hat{1})$: in fact, we have
      (remind that $n = 2, r = 1, s_1 = 1/4$ and $s_2 = 3/4$)
      \begin{align*}
       & \GG^1(\hat{1}) = \GG_{1/4}(\hat{1})
       = \Big(\frac{2^{-1/4}}{\Gamma(5/4)}\Big)^2\,(1-\|x\|^2)^{1/4}_+
       \qquad\text{and} \\
       & \qquad 
       \GG^2(\hat{1}) = \GG_{3/4}(\hat{1})
       = \Big(\frac{2^{-3/4}}{\Gamma(7/4)}\Big)^2\,(1-\|x\|^2)^{3/4}_+.
      \end{align*}
      As a consequence, we obtain
      \begin{equation} \label{eq.estimGGi}
       \begin{split}
       &  \|\GG^1(\hat{1})\|_\infty
       = \frac{1}{\sqrt{2}\,\Gamma^2(5/4)}\thickapprox 0.860682 \qquad\text{and} \\
       & \qquad 
       \|\GG^2(\hat{1})\|_\infty
       = \frac{1}{\sqrt{8}\,\Gamma^2(7/4)}\thickapprox 0.418567.
       \end{split}
      \end{equation}
	 As for the functions $\gamma_i$, the computations are much more easier:
	 first of all, since $\zeta_1\equiv\zeta_2\equiv 1$, we know from 
	 \eqref{eq.defgammaiPb} that $\gamma_i$ is the unique solution of
	 $$
	 \begin{cases}
	 (-\Delta)^{s_i}v = 0 & \text{in $B_1$}, \\
	 u \equiv 1 & \text{on $\R^2\setminus B_1$}.
	 \end{cases}
	 $$
	 On the other hand, since the above problem is solved by the constant
	 function $\hat\gamma \equiv 1$ (independently of $s_i$), we get
	 $\gamma_1 \equiv \gamma_2 \equiv 1$. Hence, we have
	 \begin{equation} \label{eq.estimgammai}
	 \|\gamma_1\|_\infty = \|\gamma_2\|_\infty = 1.
	 \end{equation}
	 Gathering together \emph{all the facts established so far},
	 we are finally in a position to apply Theorem \ref{thm.mainExistence}:
	 taking into account \eqref{eq.defrhoExmexist}, 
	 \eqref{eq.explicitM1M2}, 
	 \eqref{eq.explicitB1B2}, \eqref{eq.estimGGi} and 
	 \eqref{eq.estimgammai}, 
	 for every choice of parameters $\lambda_1,\lambda_2,\eta_1,\eta_2 \geq 0$ satisfying
	 \begin{equation} \label{eq.choiceParamExm1}
	 \begin{split}
	  \lambda_1\,\frac{\pi\cdot e}{\sqrt{2}\,\Gamma^2(5/4)}
	  + \frac{\eta_1}{2} \leq \frac{1}{2} \qquad & (\text{see assumption $\mathrm{(c)}_2$
	  with $i = 1$})\\
	  \frac{\lambda_2}{2}\,\frac{1}{\sqrt{8}\,\Gamma^2(7/4)}
	  + \frac{\eta_2}{2}\leq 1 \qquad & (\text{see assumption $\mathrm{(c)}_2$
	  with $i = 2$})\\[0.1cm]
	  \lambda_1 > 0 \qquad & (\text{see assumption $\mathrm{(c)}_1$
	  with $i_0 = 1$})
	 \end{split}
	 \end{equation}
	 there exists a solution $\bldu_0\in C^{0,1/4}(\R^2)$ of \eqref{eq.PBexmesit}, further
	 satisfying
	 $$\|\bldu\|_\infty \geq \rho_0(\delta)\qquad\text{and}\qquad
	 \|u_1\|_\infty\leq \frac{1}{2},\,\,\|u_2\|_\infty\leq 1.$$
	 Here, $\rho(\delta) \in (0,1/2)$ is as in \eqref{eq.choicerho0} and $\delta > 0$
	 is chosen in such a way that 
	 \begin{equation}\label{eq.choicedelta}
	 \delta\lambda_1 \geq \mu_1,
	 \end{equation}
	 where $\mu_1$ is the inverse of the spectral radius of $(-\Delta)^{1/4}$.
	 More explicitly, given $\lambda_1,\lambda_2,\eta_1,\eta_2$ satisfying 
	 \eqref{eq.choiceParamExm1}, one \emph{first chooses $\delta > 0$} in such a way that
	 $\delta\lambda_1\geq \mu_1$ (see assumption $\mathrm{(c)}_1$); \emph{then},
	 one lets $\rho_0 = \rho_0(\delta)$ be as in \eqref{eq.choicerho0}.
	 
	 The key point in this argument is that, since \eqref{eq.mainestimk0} holds
      \emph{for every $\delta > 0$} (by accordingly choosing $\rho_0$), 
      one is free to choose $\delta > 0$ in such a way that 
      \eqref{eq.choicedelta} holds (provided
      that $\lambda_1 > 0$), without the need of an explicit knowledge of $\mu_1$.
  \end{example}  
  \begin{example} \label{exm.NONexist}
  In Euclidean space $\R^2$, we consider the following BVP
  \begin{equation} \label{eq.PBexmNONexist}
    \begin{cases}
    (-\Delta)^{\frac{1}{4}}u_1 = \lambda_1u_1^2(1-u_1)\,\int_{B_1}e^{u_2}\,\d x & \text{in $B_1$}, \\
    (-\Delta)^{\frac{3}{4}}u_2 = \lambda_2 u_2\cdot\mathrm{osc}_{B_1}(u_1) & \text{in $B_1$}, \\
    u_1\big|_{\R^2\setminus B_1} = \eta_1\cdot u_1(0)u_2(0), \\
    u_2\big|_{\R^2\setminus B_1} = \eta_2\cdot\limsup\limits_{|x|\to\infty}u_1(x),
    \end{cases}
   \end{equation}
   where 
   $B_1$ is the Euclidean unit ball, and
   $\mathrm{osc}_{B_1}(\cdot)$ is as in \eqref{eq.defOsciexm}.
   Clearly, problem \eqref{eq.PBexmNONexist} is of the form \eqref{eq.MainDirPB}, with
   \begin{enumerate}
    \item $\Omega = B_1$, $m = 2$, $s_1 = 1/4$ and $s_2 = 3/4$;
    \item $f_1:\overline{B}_1\times\R^2\times\R,\quad 
    f_1(x,\bldz,w) := z_1^2(1-z_1)\,w$;
    \item $f_2:\overline{B}_1\times\R^2\times\R,\quad 
    f_2(x,\bldz,w) := z_2\,w$;
    \item $\PP_1:\mathbb{X}\to\R,\quad \PP_1[\bldu] := \int_{B_1}e^{u_2}\,\d x$;
    \item  $\PP_2: \mathbb{X}\to\R,\quad \PP_2[\bldu] := \mathrm{osc}_{B_1}(u_1)$;
    \item  $\BB_1:\mathbb{X}\to\R,\quad \BB_1[\bldu] := u_1(0)u_2(0)$;
    \item  $\BB_2:\mathbb{X}\to\R,\quad \BB_1[\bldu] := \limsup_{|x|\to\infty}u_1(x)$;
    \item  $\zeta_1\equiv\zeta_2\equiv 1$.
   \end{enumerate}
   Moreover, it is straightforward to recognize that
   all the `structural' assumptions (H0)-to-(H4) listed
   at the beginning of Section \ref{sec.existence} are fulfilled.
   We now aim to show that, despite the similarity
   between problems \eqref{eq.PBexmNONexist} and \eqref{eq.PBexmesit}, in this case
   it is possible to choose the parameters $\lambda_1,\lambda_2,\eta_1,\eta_2$ in such a way that
   assumptions (a)-to-(d) of the \emph{non-existence} Theorem \ref{thmnonex} are satisfied.
   \vspace*{0.05cm}
   
   To this end, we consider the (finite) sequence $\varrho$ defined as
   \begin{equation} \label{eq.choicerhoNONexist}
    \varrho := \{\rho_1,\rho_2\},\qquad
   \text{where $\rho_1 = \rho_2 = 1$};
   \end{equation}
   According to this choice of $\varrho$, we have
   (see \eqref{eq.defiIvarrho}-\eqref{eq.defconrho})
   $$I(\varrho) = [0,1]\times[0,1],\qquad
   P(\varrho) = \big\{u\in\mathbb{X}:\,\text{$\bldu(x)\in I(\varrho)$ for
   all $x\in\R^2$}\big\}.$$
       
  \noindent\textbf{Assumption (a).}
  Given any $\bldu = (u_1,u_2)\in P(\varrho)$, we have
  $$\pi\leq\PP_1[\bldu] = \int_{B_1}e^{u_2}\,\d x\leq \pi\cdot e \quad\text{and}
    \quad
    0\leq \PP_2[\bldu] = \mathrm{osc}(u_1) \leq 1;$$
    hence, assumption (a) is fulfilled with the choices
    \begin{equation} \label{eq.choiceomegaNONexist}
   \underline{\omega}_{1,\varrho} := \pi,\quad
   \overline{\omega}_{1,\varrho}
    := \pi\cdot e,\quad 
   \underline{\omega}_{2,\varrho} := 0,\quad \overline{\omega}_{2,\varrho}
    := 1.
    \end{equation}
    
    \noindent\textbf{Assumption (b).} We first observe that,
    clearly,
    both $f_1$ and $f_2$ are continuous on the whole of
    $\overline{B}_1\times\R^2\times\R$;
    mo\-re\-o\-ver, 
    we have (see 
    \eqref{eq.choicerhoNONexist} and \eqref{eq.choiceomegaNONexist})
    \begin{align*}
     & 0 \leq f_1(x,\bldz,w) = z_1^2(1-z_1)w
     \leq (\pi\cdot e)z_1\quad\text{on $\overline{B}_1\times I(\varrho)\times [\pi,\pi\cdot e]$}
     \qquad \\
     & \qquad \text{and}\quad
      0\leq f_2(x,\bldz, w) = z_2w\leq z_2
     \quad \text{on $\overline{B}_1\times I(\varrho)\times [0,1]$}.
    \end{align*}
    Thus, assumption (b) is satisfied with the choices
    \begin{equation} \label{eq.choicetauNon}
    \tau_1 = \pi\cdot e\qquad\text{and}\qquad \tau_2 = 1.
    \end{equation}
    
    \noindent\textbf{Assumption (c).} 
    Given any $\bldu = (u_1,u_2)\in P(\varrho)$, we see that
    \begin{align*}
     & 0\leq \BB_1[\bldu]
     = u_1(0)u_2(0) \leq u_1(0)\leq \|\bldu\|_\infty
     \qquad\text{and} \\
     & \qquad
     0\leq \BB_2[\bldu]
     = \limsup_{|x|\to\infty}u_1(x)\leq \|\bldu\|_\infty.
    \end{align*}
    Hence, assumption (c) is satisfied with the choices
    \begin{equation} \label{eq.choicexiNon}
    \xi_1 = \xi_2 = 1.
    \end{equation}
    
    \noindent\textbf{Assumption (d).} 
    First of all, by exploiting all the computations carried out
    in Example \ref{exm.existence} (see, respectively,
    \eqref{eq.estimGGi} and \eqref{eq.estimgammai}), we know that
    \begin{itemize}
     \item[(a)] $\|\GG^1(\hat{1})\|_\infty = \frac{1}{\sqrt{2}\,\Gamma^2(5/4)}$ and
     $\|\GG^2(\hat{1})\|_\infty = \frac{1}{\sqrt{8}\,\Gamma^2(7/4)}$;\vspace*{0.15cm}
     
     \item[(b)] $\|\gamma_1\|_\infty = \|\gamma_2\|_\infty = 1$;
    \end{itemize}
    As a consequence, by gathering together 
    \eqref{eq.choicetauNon}, \eqref{eq.choicexiNon} and the above (a)-(b),
    we can apply Theorem \ref{thmnonex}:
    for every choice of $\lambda_1,\lambda_2,\eta_1,\eta_2\geq 0$ satisfying
    $$\lambda_1\,\frac{\pi\cdot e}{\sqrt{2}\,\Gamma^2(5/4)}
    +\eta_1 < 1\qquad\text{and}\qquad
    \frac{\lambda_2}{\sqrt{8}\,\Gamma^2(7/4)}+\eta_2 < 1,$$
    the BVP \eqref{eq.PBexmNONexist} possesses only the zero solution
    in $P(\varrho)$ (notice that the constant function 
    $\bldu \equiv 0$ is indeed
    a solution of problem \eqref{eq.PBexmNONexist}).
  \end{example}


\begin{thebibliography}{xxx}

\bibitem{Alves2018} C. O. Alves, R. N. de Lima and A. B. N\'obrega, 
\emph{Bifurcation properties for a class of fractional Laplacian equations in $\mathbb R^N$},
Math. Nachr. \textbf{291} (2018), 2125--2144.

\bibitem{Alves2020} C. O. Alves, R. N. de Lima and A. B. N\'obrega, 
\emph{Global bifurcation results for a fractional equation in $\mathbb{R}^N$},
J. Math. Anal. Appl. \textbf{487} (2020), no. 1, 123980, 21 pp.

\bibitem{AlMi16}
C. O. Alves and O. H. Miyagaki,
\emph{Existence and concentration of solution for a class of fractional elliptic equation in
$\R^N$ via penalization method}, Calc. Var. Partial Differential Equations \textbf{55} (2016), n 3, Art. 47, 19 pp.

\bibitem{Amann-rev} H. Amann, 
\textit{Fixed point equations and nonlinear eigenvalue
problems in ordered Banach spaces}, SIAM. Rev. \textbf{18} (1976),
620--709.


\bibitem{AmCra} H. Amann and M. G. Crandall, 
\textit{On some existence theorems for semi-linear elliptic equations},
Indiana Univ. Math. J. \textbf{27} (1978), no. 5, 779--790.

\bibitem{Amb19}
V. Ambrosio,
\emph{Concentrating solutions for a class of nonlinear fractional Schr\"odinger equations in $\R^N$}, Rev. Mat. Iberoam. \textbf{35} (2019), no. 5, 1367--1414.

\bibitem{AmbIs18}
V. Ambrosio and T. Isernia,
\emph{A multiplicity result for a fractional Kirchhoff equation in $\mathbb{R}^N$ with a general nonlinearity}, Commun. Contemp. Math. \textbf{20} (2018), no. 5, 1750054, 17~pp. 

\bibitem{AmbSer19}
V. Ambrosio and R. Servadei, 
\emph{Supercritical fractional Kirchhoff type problems}, Fract. Calc. Appl. Anal. \textbf{22} (2019), no. 5, 1351--1377.

\bibitem{BiCaInf1}  S.\ Biagi, A.\ Calamai and G.\ Infante, 
\emph{Nonzero positive solutions of elliptic systems with gradient dependence and functional BCs}, Adv. Nonlinear Stud. \textbf{20} (2020), no. 4, 911--931. 
DOI:10.1515/ans-2020-2101.

\bibitem{BGR}
K. Bogdan, T. Grzywny and M. Ryznar,
\emph{Heat kernel estimates for the 
fractional Laplacian with Dirichlet conditions},
Ann. Probab. \textbf{38} (2010), 1901--1923.

\bibitem{Chen15}
W. Chen, \emph{Multiplicity of solutions for a fractional Kirchhoff type problem}, Commun. Pure Appl. Anal. \textbf{14} (2015), no. 5, 2009--2020. 

\bibitem{CM2020}
A. C. R. Costa and B. B. V. Maia, 
\emph{Multiple solutions of systems involving fractional Kirchhoff-type equations with critical growth}, Differ. Equ. Appl. \textbf{12} (2020), no. 2, 165--184.

\bibitem{DNPV} 
E. Di Nezza, G. Palatucci and E. Valdinoci, 
\emph{Hitchhiker's guide to the fractional Sobolev spaces}, 
Bull. Sci. Math. \textbf{136} (2012), 521--573. 

\bibitem{ERGPS}
G. Estrada-Rodriguez, H. Gimperlein, K. J. Painter and J. Stocek, \emph{Space-time fractional diffusion in cell movement models with delay}, Math. Models Methods Appl. Sci. \textbf{29} (2019), no. 1, 65--88.


\bibitem{FSV}
A. Fiscella, R. Servadei and E. Valdinoci, 
\emph{Density properties for fractional Sobolev spaces}, 
Ann. Acad. Sci. Fenn. Math. \textbf{40} (2015), 235--253.

\bibitem{FV14}
A. Fiscella and E. Valdinoci, 
\emph{ A critical Kirchhoff type problem involving a nonlocal operator},
Nonlinear Anal. \textbf{94} (2014), 156--170.

\bibitem{Get}
R. K. Getoor,
\emph{First passage times for symmetric stable 
processes in space},
Trans. Amer. Math. Soc. \textbf{101} (1961), 75--90.

\bibitem{guolak} D. Guo and V. Lakshmikantham,
\textit{Nonlinear problems in abstract cones}, Academic Press, Boston,
(1988). 

\bibitem{HB}
T. Hillen and A. Buttensch\"on, \emph{Nonlocal adhesion models for microorganisms on bounded domains}, SIAM J. Appl. Math. \textbf{80} (2020), no. 1, 382--401.


\bibitem{gi-tmna} G. Infante, \textit{Nonzero positive solutions of a multi-parameter elliptic system with functional BCs}, Topol. Methods Nonlinear Anal. \textbf{52} (2018), 665--675.

\bibitem{gi-jepe} G. Infante, \textit{Nonzero positive solutions of nonlocal elliptic systems with functional BCs}, J. Elliptic Parabol. Equ. \textbf{5} (2019), 493--505.

\bibitem{gi-BK}
G. Infante, \textit{Eigenvalues of elliptic functional differential systems via a Birkhoff-Kellogg type theorem}, Mathematics, to appear.

\bibitem{LK}
G. Lemon and J. R. King, \emph{A functional differential equation model for biological cell sorting due to differential adhesion}, Math. Models Methods Appl. Sci. \textbf{23} (2013), no. 1, 93--126.


\bibitem{PST}
B. Perthame, W. Sun, and M. Tang,
\emph{The fractional diffusion limit of a kinetic model with biochemical pathway}, Z. Angew. Math. Phys. \textbf{69} (2018), no. 3, Paper No. 67, 15 pp.


\bibitem{RosSerra} 
X. Ros-Oton and J. Serra, 
\emph{The Dirichlet problem for the fractional Laplacian: regularity up to the boundary}, 
J. Math. Pures Appl. (9) \textbf{101} (2014), 275--302. 

\bibitem{Secchi}
S. Secchi,
\emph{Ground state solutions for nonlinear fractional Schr\"odinger equations in $\R^N$}, J. Math. Phys. \textbf{54} (2013), 031501-17 pages.

\bibitem{SerVal}
R. Servadei and E. Valdinoci,
\emph{Variational methods for non-local operators of elliptic type},
Discrete Contin. Dyn. Syst. \textbf{33} (2013), 2105--2137.

\bibitem{SerVal2}
R. Servadei and E. Valdinoci,
\emph{A Brezis-Nirenberg result for non-local critical equations in low dimension}, 
Commun. Pure Appl. Anal. \textbf{12} (2013), 2445--2464.

\bibitem{SunTeng14}
G. F. Sun and K. M. Teng, 
\emph{Existence and multiplicity of solutions for a class of fractional Kirchhoff-type problem}, Math. Commun. \textbf{19} (2014), 183--194.

\end{thebibliography}
\end{document}